\documentclass[11pt,righttag]{amsart}
\usepackage{diagrams}
\usepackage{graphicx}

\newtheorem{thm}{Theorem}[section]
\newtheorem{lemma}[thm]{Lemma}
\newtheorem{cor}[thm]{Corollary}

\theoremstyle{definition}
\newtheorem{defn}[thm]{Definition}

\newtheorem{rmk}[thm]{Remark}

\newtheorem{example}[thm]{Example}

\newcommand{\N}{{\mathbb N}}

\newcommand{\F}{{\mathbb F}}

\newcommand{\NKD}{{\rm NKD}}

\newcommand{\Ext}{\hbox{{\rm Ext}}}

\begin{document}


\title[Splitting Algebras]{Splitting Algebras: Koszul, Cohen-Macaulay and numerically Koszul}


\author[  Kloefkorn, Shelton ]{Tyler Kloefkorn and Brad Shelton}
\address{University of Oregon\\
Eugene, Oregon 97403}
\email{tkleofko@uoregon.edu, shelton@uoregon.edu}

%
%
%
%
%
%
%


\begin{abstract}
\baselineskip12pt
We study a finite dimensional quadratic graded algebra $R_\Gamma$ defined from a finite ranked poset $\Gamma$. This algebra has been central to the study of the {\it splitting algebras} $A_\Gamma$  introduced by Gelfand, Retakh, Serconek and Wilson, \cite{GRSW}. Those algebras are known to be quadratic when $\Gamma$ satisfies a combinatorial condition known as {\it uniform}.  A central question in this theory has been: when are the algebras Koszul?  We prove that  $R_\Gamma$ is Koszul and $\Gamma$ is uniform if and only if the poset $\Gamma$ is Cohen-Macaulay.  We also show that the cohomology of the order complex of $\Gamma$ can be identified with certain cohomology groups defined internally to the ring $R_\Gamma$, $H_{R_\Gamma}(n,0)$ (introduced in \cite{CPS}) whenever $\Gamma$ is Cohen-Macualay.  Finally, we settle in the negative the long-standing question: Does numerically Koszul imply Koszul for algebras of the form $R_\Gamma$.  
\end{abstract}
\date{July 6, 2012}

\maketitle



\baselineskip18pt

\section{Introduction}

We fix, once and for ever, a field $\F$.  All topological cohomology groups are calculated with coefficients in $\F$.  

Let $\Gamma$ be a finite ranked poset with unique minimal element $*$ and strict order $<$.  Write $x\to y$ if $x$ covers $y$ in the usual sense.  The $\F$-algebra $R_\Gamma$ has generators $r_x$ for $*\ne x\in \Gamma$ and relations
$$ r_xr_y=0 \hbox{ if } x \not \to y\quad  \hbox{\ and \ } \quad r_x\sum\limits_{x\to y} r_y=0.$$
In the literature this algebra was denoted $B(\Gamma)$ in \cite{RSW3} and \cite{RSW4} and $R(\Gamma)$ in \cite{CPS}.

The quadratic dual of this algebra, $A'_\Gamma=R_\Gamma^!$ also has generators $r_x$ for $*\ne x\in \Gamma$ and relations
$$ r_x r_y = r_x r_z\ \  \hbox{\ whenever \ }\ \  x\to y\  \hbox{\ and \ }x\to z.$$
If $\Gamma$ satisfies a simple combinatorial property known as {\it uniform} (see \ref{uniform}) then $A'_\Gamma$ is a deformation of the 
{\it splitting algebra} $A_\Gamma = A(\Gamma)$ of the poset $\Gamma$, as introduced in \cite{GRSW} (the phrase splitting algebra did 
not come into use until \cite{RSW3} and \cite{CPS}).  Splitting algebras are related to the problem of factoring non-commuting polynomials.  

The two main problems related to these various algebras have been: A) calculate the Hilbert series and B) determine exactly when 
the algebras satisfy the Koszul property.  The first question was answered in the papers \cite{GGRSW}, \cite{RSW1} and \cite{RSW4}.  
Some progress was made on the second question in \cite{CPS}, with a subset of those same results appearing in \cite{RSW3}.  This paper substantially improves on those results.

The paper \cite{CPS} introduced certain cohomology groups internal to the ring $R_\Gamma$: $H_{R_\Gamma}(n,k)$ 
(see \ref{cohomology}).  Our first main theorem shows that these cohomology groups, when $R_\Gamma$ is Koszul, capture all of 
the cohomology of the topological space associated to the order complex of $\Gamma$.

\begin{thm}\label{intro1}
Let  $\Gamma$ be a uniform ranked poset for which $R_\Gamma$ is Koszul and let $Y$ be the total space of the order complex of 
$\Gamma\setminus\{*\}$.  
Then  $H^n(Y) = H_{R_\Gamma}(n,0)$ for all $n$.  
\end{thm}

Our second main theorem classifies all ranked posets that are uniform and for which $R_\Gamma$ is Koszul.  Recall that a poset is 
Cohen-Macaulay relative to a field $\F$ if and only if the order complex of any open subinterval $(a,b)$ has non-zero reduced homology 
(or equivalently cohomology) only in the degree equal to its dimension.  

\begin{thm}\label{intro2}Let $\Gamma$ be a finite ranked poset. Then $\Gamma$ is uniform and $R_\Gamma$ is Koszul if and only if 
$\Gamma$ is Cohen-Macaulay. 
\end{thm}

We note that $A_\Gamma'$ is Koszul if and only if $R_\Gamma$ is Koszul and that either of these algebras being Koszul implies 
$A_\Gamma$ is Koszul (\cite{PP}, 4.7.1).  It is not known to us at this time if the Koszul property for $A_\Gamma$ implies the Koszul 
property for $A_\Gamma'$.  Hence the theorem only answer the Koszul question for $A_\Gamma$ in one direction. 

We further note that the Cohen-Macaulay property is relative to the field $\F$ and hence there exist many examples of posets $\Gamma$ for which the Koszul property of $R_\Gamma$ is relative to $\F$ as well. 

Another outstanding question about the algebras $R_\Gamma$ relates to numerical Koszulity.  A quadratic algebra $R$, with quadratic dual algebra $R^!$, and Hilbert series $H(R,t)$, is said to be {\it numerically Koszul} if
$$H(R,-t)\cdot H(R^!,t) = 1.$$  
All Koszul algebras are numerically Koszul, but the converse is generally false.  It was not known previously if the algebra 
$R_\Gamma$ could be numerically Koszul without being Koszul.  This is answered by the following theorem.

\begin{thm}\label{intro3}
There exist uniform ranked posets $\Gamma$, including cyclic posets, for which the algebra $R_\Gamma$ is numerically Koszul but 
not Koszul.
\end{thm}

We give a general construction of such examples and use that construction to provide a non-cyclic example in \ref{main3}.  By 
ad-hoc methods we give one specific cyclic example in \ref{wild}.  

\medskip
The defining relations of the algebras $R_\Gamma$ and $A'_\Gamma$  are extremely simple. We find it intriguing that these algebras 
satisfy a theorem analogous to the many classical (see for example \cite{BGSsurvey}, \cite{Reisner} and \cite{Stanley75}) and 
neo-classical (see for example \cite{Polo} and \cite{Woodcock}) theorems that associate the Cohen-Macaulay property to some algebraic 
or homological property of a class of rings.  

The outline of our paper is as follows:  Section 2 provides basic definitions such as uniform and 
$H_{R_\Gamma}(n,k)$ and reviews some results from \cite{CPS}.  Section 3 reviews the definitions of the order complex and the 
associated cohomology groups. Theorem \ref{intro1} is then proved as Theorem \ref{main1}.  Section 5 is a short digression to 
connect our results to those of \cite{RSW4}.  Theorem \ref{intro2} is proved as Theorem \ref{main2}.  We discuss a few  
examples in Section 7. Section 8 of the paper is quite technical and intended for those readers interested in the question of numerical 
Koszulity.  In that section we prove Theorem \ref{intro3} via \ref{main3} and \ref{wild}.

\section{Definitions and preliminaries}

We start with basic notation related to ranked posets.

\begin{defn}  
Let $\Gamma$ be a poset with unique minimal element $*$ and strict order $<$.  We say $\Gamma$ is ranked if for all $b\in \Gamma$, any two maximal chains in $[*,b]$ have the same length.  The length of such a maximal chain is then referred to as the  {\it rank} of $b$ and written $rk_\Gamma(b)$.   Let $\Gamma(k)$ be the elements of $\Gamma$ of rank $k$.  

(1) $\Gamma$ is pure of rank $d$ if $rk_\Gamma(x)= d$ for every maximal element of $\Gamma$.  

(2) $\Gamma_x$ denotes the interval $[*,x]$ in $\Gamma$.

(3) $\Gamma$ is cyclic if $\Gamma = \Gamma_x$ for some $x\in \Gamma$.

(4) For any $x\in \Gamma$, $S_x(k) = \{ y\in \Gamma_x\, |\, rk_\Gamma(y) = rk_\Gamma(x) -k \}$.

\noindent For any $a<b$, we say that $b$ covers $a$, written $b\to a$, if the closed interval $[a,b]$ has order 2, or equivalently $a\in S_b(1)$. This makes $\Gamma$ into a directed graph that is often referred to as a {\it layered} graph.
\end{defn}

We recall the definition of {\it uniform} from \cite{GRSW}.

\begin{defn}\label{uniform} 
Let $\Gamma$ be a ranked poset.  For $x\in \Gamma$ and $a,b \in S_x(1)$, write
$a\sim_x b$ if there exists $c\in S_a(1) \cap S_b(1)$ and extend $\sim_x$ to an equivalence relation on $S_x(1)$.   We say that 
$\Gamma$ is {\it uniform} if, for every $x\in \Gamma$, $\sim_x$ has a unique equivalence class.  
\end{defn}

We next define the $\F$-algebras we are interested in studying. 

\begin{defn}
Let $\Gamma$ be a finite ranked poset.  Let $V_\Gamma$ be the $\F$-vector space with basis elements 
$r_x$, $*\ne x\in \Gamma$.  For each $k\ge 0$ and  $*\ne x\in \Gamma$, set $r_x(k) = \sum\limits_{y\in S_x(k)} r_y$.
Let $I_\Gamma$ be the quadratic ideal of the free (tensor) algebra $\F(V_\Gamma)$ generated by the elements:

(1) $r_x \otimes r_y$ for all pairs $\{x,y\}$ such that $y\not\in S_x(1)$,

(2) $r_x \otimes r_x(1)$ for all $x$.

\noindent Then $R_\Gamma$ is the $\F$-algebra $\F(V_\Gamma)/I_\Gamma$ and we continue to write $r_x$ for the 
generators of $R_\Gamma$.
\end{defn}

The algebra $R_\Gamma$ can be graded in several convenient ways, but we will only use the standard connected grading
 $R_\Gamma= \bigoplus\limits_{n\ge 0}R_{\Gamma,n}$ in which the generators $r_x$ have degree 1.  One fundamental observation about $R_\Gamma$ (cf.~\cite{CPS}) is:
$$ R_{\Gamma,+}=  \bigoplus\limits_{*\ne x\in \Gamma} r_x R_\Gamma.$$
We will use this observation repeatedly and without further comment.

Recall that a graded connected $\F$-algebra is said to be Koszul if the associated bigraded Yoneda algebra $\Ext_R^{*,*}(\F,\F)$ is generated  by $\Ext_R^{1,1}(\F,\F)$ as an $\F$-algebra.  Koszul algebras must be quadratic and there are many equivalent ways to define them (see for example \cite{PP}).

We will say that a finite ranked poset $\Gamma$ is Koszul if the algebra $R_\Gamma$ is Koszul.  We warn the reader that this is an abuse of notation since we know from \cite{CPS} that this definition is dependent on the field $\F$. That is, there are posets $\Gamma$ such that the property ``$R_\Gamma$ is Koszul'' is dependent on the field $\F$.  

The following result from 
\cite{CPS} is extremely useful and will be used repeatedly without further comment.

\begin{thm}\label{useful} If the poset $\Gamma$ is uniform, then 
$\Gamma$ is Koszul if and only if $\Gamma_x$ is Koszul for every $*\ne x\in \Gamma$.
\end{thm}

The Koszul property is closely related to certain cohomology groups built from the ring $R_\Gamma$.  We describe these in the next two definitions.

\begin{defn}
Let $\Gamma$ be a finite uniform ranked poset.

(1) $d_\Gamma = \sum\limits_{*\ne x\in \Gamma} r_x \in R_{\Gamma,1}$.  Also let $d_\Gamma$ denote the function 
$d_\Gamma: R_\Gamma \to R_\Gamma$ given by {\bf left} (but never right) multiplication by $d_\Gamma$.  

(2) For all $n\ge k\ge 0$, set $R_\Gamma(n,k) = \sum\limits_{rk_\Gamma(y) = n+1} r_y R_{\Gamma,n-k}$.

\end{defn}

From the definition of $R_\Gamma$ it is clear that $R_\Gamma$ has a spanning set of monomials of the form $r_{b_1}r_{b_2}\cdots r_{b_j}$ where $b_1\to b_2\to \cdots \to b_j$.   The space $R_\Gamma(n,k)$ is then the span of such monomials for which $rk_\Gamma(b_1) = n+1$ and $rk_\Gamma(b_j) = k+1$. The degree of such a monomial is $n-k+1$. 

From the definitions we see at once that $(d_\Gamma)^2 = 0$.  In particular, for each $k\ge 0$ we have a cochain complex:
$$ \cdots R_\Gamma(n-1,k) \xrightarrow{d_\Gamma} R_\Gamma(n,k)  \xrightarrow{d_\Gamma} R_\Gamma(n+1,k) \cdots $$
It is useful to note the folllowing.  Let $d_\Gamma^n = \sum\limits_{rk(y) = n+1}r_y$, so that $d_\Gamma = \sum_n d_\Gamma^n$. Then the cochain complex above is the same as:
$$ \cdots R_\Gamma(n-1,k) \xrightarrow{d_\Gamma^n} R_\Gamma(n,k)  \xrightarrow{d_\Gamma^{n+1}} R_\Gamma(n+1,k) \cdots $$

\begin{defn}\label{cohomology}
For each $k\ge 0$, we will denote the cohomology of the complex above, $H^n(R_\Gamma(\cdot,k),d_\Gamma)$, by $H_{R_\Gamma}(n,k)$, or more simply as $H_\Gamma(n,k)$.    
\end{defn}

It is sometimes convenient to augment each of the cochain complexes 
$R_\Gamma(\cdot,k)$ by defining  $\F \to R(k,k)$ via $1\mapsto d_\Gamma^k$.  We denote the cohomology of the augemented complex 
by $\tilde H_\Gamma(n,k)$. Please note that this differs from $H_\Gamma(n,k)$ in cohomology degree $k$, not 0. 

The following important theorem from \cite{CPS} explains how the internal cohomology groups $H_\Gamma$ are related to the Koszul property of the algebra 
$R_\Gamma$. This is the primary inductive tool for lifting the Koszul property from rank $d$ to rank $d+1$.

\begin{thm}\label{cyclic}
Assume $\Gamma=\Gamma_x$ is a cyclic uniform ranked poset of with $rk_\Gamma(x)=d+1$.   Then

{\rm (1)} $H_\Gamma(k,k) = \F$ for all $0\le k\le d$.

{\rm (2)} $H_\Gamma(n,k) = 0$ if $n =d$ or $d-1$ and $k<n$.

{\rm (3)} Assume $\Gamma_z$ is Koszul for every $z<x$.  Then $\Gamma$ is Koszul if and only if $H_\Gamma(n,k) = 0$ for all $0\le k<n\le d-2$.  

\end{thm}
 
 \begin{proof}  This is 3.7 and 3.8 of \cite{CPS}.
 \end{proof}

\begin{rmk} Despite a remark to the contrary in \cite{CPS}, neither the cyclic hypothesis nor the ``inductive'' hypothesis can be removed from part (3) of \ref{cyclic}.    We give two examples below to illustrate these points.  These examples are not necessary to the rest of the paper. 
\end{rmk}

\begin{example}\label{ex1}  Historically, the first known example of a poset for which $R_\Gamma$ is not Koszul is the poset $\Gamma$ whose Hasse diagram is  shown in the figure below.  Let $\Gamma' = \Gamma\setminus \{X\}$.  One sees directly that the element $r_Dr_G$ in 
$R_{\Gamma'}(1,0) = R_\Gamma(1,0)$ represents a non-zero cohomology class in both $H_{\Gamma'}(1,0)$ and 
$H_{\Gamma}(1,0)$.  Since $R_{\Gamma'}$ is Koszul (because all rank 3 cases are Koszul), this shows that Koszulity, without the cyclic hypothesis, does not guarantee vanishing of cohomology.  The fact that $H_{\Gamma}(1,0)$ is non-zero does, however, prove that $R_\Gamma$ is not Koszul (by (3) of \ref{cyclic}).

\medskip
\begin{center}

\includegraphics[width=1in]{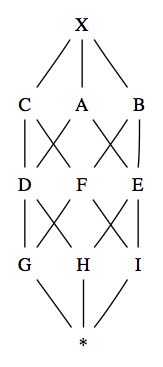}

\end{center}
 
\medskip

\end{example}

It is rather more complex to show that without the inductive hypothesis vanishing of the cohomology groups $H_{R_\Gamma}(n,k)$ for $k<n$ does not imply the Koszul property.  We use the results of \cite{CPS} to build a fairly straightforward cyclic example.

\begin{example}\label{ex2}  Let ${\bf Z}$ and ${\bf Y}$ be two regular CW complexes pictured to the left below and let $\Omega$ be the uniform ranked poset whose Hasse diagram is given to the right below.  

\medskip
\begin{center}

\includegraphics[width=1.6in]{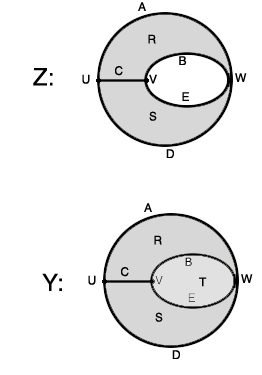}\hskip .4in
\includegraphics[width=1.6in]{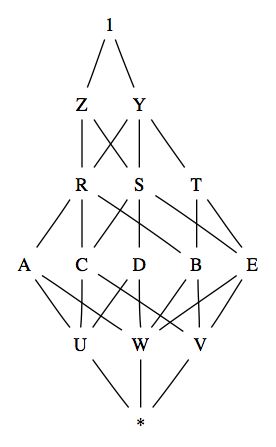}

\end{center}
 
\medskip
We claim that $H_\Omega(n,k)=0$ for all $0\le k<n\le 4$, but $\Omega$ is not Koszul.  

First note that $\Omega_{Z}$ has the form $P\cup \{*,Z\}$, where $P$ is the incidence poset of the CW complex ${\bf Z}$.   Since ${\bf Z}$ is homotopic to $S^1$, but is pure of dimension 2, Corollary 5.6 of \cite{CPS} tells us $\Omega_Z$ is not Koszul.  Hence $\Omega$ is not Koszul. 

On the other hand, $\Omega_Y$ has the form $Q\cup\{*,Y\}$, where $Q$ is the incidence poset of ${\bf Y}$. Since ${\bf Y}$ is a 2-disc, Corollary 5.6 of \cite{CPS}  tells us that $\Omega_{Y}$ is Koszul and then \ref{cyclic} tells us  that 
$H_{\Omega_{Y}}(n,k) = 0$ for all $0\le k<n\le 3$.  Since the element $Y$ majorizes every element of 
$\Omega$ of rank at most 3, and since $H_{\Omega_Y}(2,k)=0$ for $k=0,1$, one can see by inspection that 
$H_\Omega(n,k) = H_{\Omega_{Y}}(n,k)$ for all $0\le k\le n\le 2$.  Combining this with (2) of \ref{cyclic} shows that $H_\Omega(n,k) = 0$ for all $0\le k<n\le 4$, as claimed.
\end{example}

\section{The Order Complex of a Poset}

The notion of the order complex of a finite poset is a standard tool in combinatorial topology and elsewhere.  For completeness of exposition, we include the basic definitions.    

\begin{defn}   Let $\Gamma$ be a finite poset with strict order $<$.  The {\it order complex} of $\Gamma$, $\Delta(\Gamma)$, is the collection of ordered subsets of $\Gamma$:
$$\Delta(\Gamma)= \{ (b_0,b_1,\ldots,b_n)\,|\,  b_i \in \Gamma \hbox{\ and \ } b_0<b_1<\cdots< b_n\}$$ 
 An element $\beta=(b_0,\cdots,b_n)$ in $\Delta(\Gamma)$ is an $n$-cell (or $n$-chain) of the complex, 
 $C^n(\Delta(\Gamma))$ denotes the $\F$-vector space generated by the $n$-cells and $C(\Delta(\Gamma)) = \oplus_nC^n(\Delta(\Gamma))$.  Given $x\in \Gamma$, define 
 $u_x:C^n(\Delta(\Gamma))\to C^{n+1}(\Delta(\Gamma))$ by extending linearly from the formula
$$u_x(b_0,\cdots,b_n) = \begin{cases} 
		(x,b_0,\cdots,b_n) & \hbox{if\ } x<b_0\\ 
		(-1)^{i+1} (b_0,\cdots,b_i,x,b_{i+1},\cdots,b_n) & \hbox{if\ } b_i<x<b_{i+1}\\ 
		(-1)^{n+1}(b_0,\cdots, b_n,x) & \hbox{if\ } b_n<x\\
		0 & \hbox{otherwise\ } \\  
\end{cases}$$
Finally, we set $d_{\Delta(\Gamma)} = d = \sum_x u_x:C(\Delta(\Gamma)) \to C(\Delta(\Gamma))$.
\end{defn}

Since $\Delta(\Gamma)$ is a simplicial complex, it has a well-defined geometric realization, or {\it total space} which we will denote $||\Delta(\Gamma)||$.  We typically give this topological space a name, say $Y = ||\Delta(\Gamma)||$.   We will consistently abuse notation and write $C(Y)$ for $C(\Delta(\Gamma))$ and $d_Y$ for $d_{\Delta(\Gamma)}$.  It is standard that  $(C(Y),d_Y)$ is a cochain complex and that $H^n(Y) = H^n(C(Y),d_Y)$.   We remind the reader that these cohomology groups are all calculated with coefficients in our base field, $\F$.

\section{The First Main Theorem}

Let $\Gamma$ be a finite ranked poset with unique minimal element $*$.   
Let $Y$ be the total space of the order complex $\Delta(\Gamma\setminus\{*\})$.  We define an epimorphism:
$$\Phi_\Gamma: C^n(Y) \to R_\Gamma(n,0)$$
by extending linearly from the formula:
$$ \Phi_\Gamma((b_0,\cdots,b_n)) = \begin{cases}
		r_{b_n}r_{b_{n-1}}\cdots r_{b_0} & \hbox{ if\ } rk_\Gamma(b_0)=1 \\
		0 & \hbox{\ otherwise\ } \\
\end{cases}
$$
We note that $\Phi_\Gamma((b_0,\cdots,b_n)) = 0$ unless $rk_\Gamma(b_n)= n+1$, since otherwise there is some $j$ for which $b_j \not\to b_{j-1}$ and then $r_{b_j}r_{b_{j-1}} =0$.  It may seem odd to utilize a map that annihilates so much information, but it works.

\begin{lemma}  $\Phi_\Gamma: C(Y) \to R_\Gamma(\cdot,0)$ is a cochain epimorphism.
\end{lemma}

\begin{proof}  We begin with a preliminary observation.  Fix $*\ne a<b$ in $\Gamma$.  We claim: 
$\sum\limits_{a<x<b} r_br_xr_a = 0$ in $R_\Gamma$.  To see this, first observe that we may eliminate from the sum any $x$ that is not in $S_b(1)$ since for such $x$, $r_br_x=0$.  On the other hand, for any $y\in S_b(1)$ for which $a\not<y$, we have $r_yr_a = 0$.  Hence we may add such terms to the sum without changing it.  Hence the sum is the same as $r_b r_b(1) r_a$, which is $0$ by definition.   A similar observation is that $\sum\limits_{*\ne x<b} r_b r_x = 0$.  

Fix $\beta=(b_0,\cdots,b_n)\in C^n(Y)$. Consider first the case when $rk(b_0)>1$, in which case $d_\Gamma \Phi_\Gamma(\beta) = d_\Gamma 0 = 0$.  Then 
$$ \Phi_\Gamma(d_Y(\beta)) =  \sum\limits_{*\ne x<b_0} \Phi_\Gamma(x,b_0,\cdots, b_n) = \sum_{x<b_0,\, rk(x)=1} r_{b_n}\cdots r_{b_0} r_x.$$
If $rk(b_0) >2$ then this sum is $0$ since every term $r_{b_0} r_x = 0$.  If $rk(b_0) = 2$, then this sum is $0$ by the observation above.  Either way, $\Phi_\Gamma(d_Y(\beta)) = (-1)^{n+1} d_\Gamma \Phi_\Gamma(\beta) = 0$.

Consider the case when $rk(b_0) = 1$.  Recalling that $r_yr_{b_n} = 0$ whenever $b_n\not<y$, we get
$$d_\Gamma \Phi_\Gamma(\beta) = d_\Gamma r_{b_n}\cdots r_{b_0} = \sum\limits_{b_n<y} r_yr_{b_n}\cdots r_{b_0}.$$
On the other hand,
\begin{align*}
\Phi_\Gamma(d_Y(\beta)) &=  \sum_i (-1)^{i+1}\sum\limits_{b_i< x<b_{i+1}}  r_{b_n}\cdots r_{b_{i+1}}r_xr_{b_i}\cdots r_{b_0}\\
		&\ \ \ \ \ \ + (-1)^{n+1}\sum\limits_{b_n<x} r_xr_{b_n}\cdots r_{b_0}.
\end{align*}
By the observation above, each sum inside the first term of this expression is $0$.  This shows that $\Phi_\Gamma(d_Y(\beta)) = (-1)^{n+1} d_\Gamma \Phi_\Gamma(\beta)$, as required.
\end{proof}

It is clear that the cochain map $\Phi_\Gamma$ extends to a cochain map between the augmented cochains 
$\F \to C(Y)$ and $\F\to R_\Gamma(\cdot,0)$.

Now we state and prove a more precise version of Theorem \ref{intro1}.

\begin{thm}\label{main1}  Let $\Gamma$ be a finite uniform ranked poset  and $Y=||\Delta(\Gamma\setminus\{*\})||$.  
Assume $R_\Gamma$ is Koszul.   Then the cochain map $\Phi_\Gamma:C^n(Y) \to R_\Gamma(n,0)$ is a quasi-isomorphism.  In particular:
$$ H^n(Y) = H_\Gamma(n,0) \hbox{ for all } n.$$

\end{thm}

\begin{proof}  Let $R=R_\Gamma$ throughout the proof.  Let $d+1$ be the maximal rank of any element of $\Gamma$.  We prove the theorem by induction on $d$.  The case $d=0$ is clear.  Henceforth we assume $d>0$.  We begin by proving a special case of the theorem.  The special case contains substantive extra information.

\begin{lemma}\label{magic}   Let $\Gamma$, $Y$ be as in \ref{main1}, with the additional hypothesis that $\Gamma$ is cyclic, that is
$\Gamma = \Gamma_x$ where $rk_\Gamma(x) = d+1$.   Let $\Gamma' = \Gamma\setminus\{x\}$, let $Z$ be the $(d-1)$-dimensional closed subspace of $Y$ given by $Z = ||\Delta(\Gamma'\setminus\{*\})|| = ||\Delta((*,x))||$.  

{\rm (1)}  $\Phi_\Gamma:C^n(Y) \to R(n,0)$ is a quasi-isomorphism,

{\rm (2)} $\tilde H^n(Y) = \tilde H_\Gamma(n,0) = 0$ for all $n$,

{\rm (3)}  $\tilde H^n(Z) = 0$ for  all $n\ne d-1$,

{\rm (4)} The map $\tilde H^{d-1}(Z) \to R(d,0)$ given by 
$$[(b_0,\cdots,b_{d-1})] \mapsto r_x\Phi_{\Gamma'}(b_0,\cdots,b_{d-1}) = r_xr_{b_{d-1}}\cdots r_{b_0}$$
is an isomorphism. 

\end{lemma}

\begin{proof}
Since $R=R_\Gamma$ is by assumption Koszul, \ref{cyclic} tells us $ \tilde H_\Gamma(n,0) = 0$ for all $n$.  Since $\Gamma$ is cyclic, the space $Y$ is contractible and thus
$\tilde H^n(Y) = 0$ for all $n$.  This proves (2), from which (1) follows trivially. 

Let $R' = R_{\Gamma'}$.  By induction, $\Phi_{\Gamma'}:C^n(Z) \to R'(n,0)$ is a quasi-isomorphism and $\tilde H^n(Z) = \tilde H_{\Gamma'}(n,0)$ for all $n$.  

Define the cochain complex $K$ to be $0\to R(d,0) \to 0$ with the term $R(d,0)$ in degree $d$. We note that $R(n,0) = R'(n,0)$ for all $n<d$, and $R'(d,0)=0$.  Moreover, the maps $d_\Gamma$ and $d_{\Gamma'}$ coincide on the spaces $R(n,0)$ for $n<d-1$. Hence we have short exact sequence of cochains:
$$ 0 \to K \to R(\cdot,0) \to R'(\cdot,0) \to 0$$
The associated long exact sequence in cohomology, together with (2),  yields $H_{\Gamma'}(n,0) =H_{\Gamma}(n,0)= 0$ for all $n<d-1$. Furthermore,  $ H_{\Gamma'}(d-1,0)$ is isomorphic to $R(d,0)$ via the connecting homomorphism.  Composing the connecting homomorphism with the isomorphism $\Phi_{\Gamma'}:\tilde H^{d-1}(Z) \to \tilde H_{\Gamma'}(d-1,0)$ gives exactly the map given in (4).  This proves (3) and (4).
\end{proof}

We return to proving the general case of the Theorem.   Let $\Gamma(d+1) = \{y_1,\ldots,y_s\}$ and set $\Omega = \Gamma \setminus \Gamma(d+1)$.  We define closed subspaces of $Y$:  $Y_i = ||\Delta((*,y_i])||$ for $1\le i\le s$ and 
$Z =||\Delta(\Omega\setminus\{*\})||$.  Note that for each $i$,  $Z\cap Y_i = ||\Delta((*,y_i))||$, so that Lemma \ref{magic} applies to the pair $(Y_i,Z\cap Y_i)$.   

Consider the relative cochain complex $C(Y,Z)$.  The basis elements of $C^{n+1}(Y,Z)$ are those 
$n+1$-cells $(b_0,\cdots,b_{n+1})$ in $C^{n+1}(Y)$ for which $b_{n+1} = y_i$ for some $i$, in 
which case $(b_0,\cdots,b_n)$ is in $C^n(Z\cap Y_i)$.  Hence there is a vector space isomorphism 
$\zeta: \oplus_i C^n(Z\cap Y_i) \to C^{n+1}(Y,Z)$ given by mapping $(b_0,\cdots,b_n)$ in 
$C^n(Z\cap Y_i)$ to $(b_0,\ldots,b_n,y_i)$.  We also define the isomorphism 
$\zeta: \F^s \to C^0(Y,Z)$ by $\zeta(e_i) = (y_i)$.  Finally, define an augmentation 
$\F^s \to \oplus_iC^0(Z\cap Y_i)$ via $e_i \mapsto \sum\limits_{b<y_i} (b)$ in $C^0(Z\cap Y_i)$.  

Using the fact that each $y_i$ is maximal in $\Gamma$, it is a straightforward calculation to see that $\zeta$ is a degree +1 cochain map between the augmented cochain complex $\F^s \to \oplus_i C(Z\cap Y_i)$ and the complex $C(Y,Z)$.  Hence $\zeta$ is an isomorphism of cochain complexes.  Thus 
$H^{n+1}(Y,Z) = \bigoplus_i \tilde H^n(Z\cap Y_i)$ for all $n$. By Lemma \ref{magic}, we then 
have $H^n(Y,Z) = 0$ for all $n<d$.  Since $R(d,0) = \bigoplus_i r_{y_i}R(d-1,0) = \bigoplus_i R_{\Gamma_{y_i}} (d,0)$, \ref{magic} also shows us that $H^d(Y,Z)$ is isomorphic to $R(d,0)$, via the map $[\beta=(b_0,\cdots,b_{d-1},y_i)] \mapsto r_{y_i}r_{b_{d-1}}\cdots r_{b_0} = \Phi_\Gamma(\beta)$.  

Let $K$ be the cochain complex $0\to R(d,0) \to 0$, concentrated in degree $d$.  Exactly as in the proof of \ref{magic} we have a short exact sequence of cochain complexes:  $0\to K\to R(\cdot,0) \to R_\Omega(\cdot,0) \to 0$.   For any $n<d$, let $\hat \Phi_\Gamma$ be the restriction of 
$\Phi_\Gamma$ to  $C^n(Y,Z)$.  By the note just after the definition of $\Phi_\Gamma$, we see 
$\hat\Phi_\Gamma(C^n(Y,Z)) = 0$ for all $n<d$.  
   
Using the last observation, we see that we have a commutative diagram of cochain complexes:
\begin{diagram}
0&  \longrightarrow&C(Y,Z) & \longrightarrow&C(Y)  & \longrightarrow&C(Z) & \longrightarrow & 0 \\
&& \dTo_{\hat\Phi_\Gamma} &&\dTo_{\Phi_\Gamma} && \dTo_{\Phi_\Omega}\\
0& \longrightarrow&K & \longrightarrow&R(\cdot,0)  & \longrightarrow&R_\Omega(\cdot,0) & \longrightarrow & 0 \\
\end{diagram}
By induction, $\Phi_\Omega$ is a quasi-isomorphism.  By the previous paragraph, 
$\hat\Phi_\Gamma$ is  a quasi-isomorphism.  Thus $\Phi_\Gamma$ is a quasi-isomorphism.  This completes the proof of \ref{main1}
\end{proof}

\section{Connection to a Theorem of Retakh, Serconek and Wilson}

This section is a brief digression in order make a connection between our methods and a very good result: Proposition 3.2.1 
of \cite{RSW4}.  The basic idea of this section is to see just how far one can push the techniques of the previous section without 
the Koszul hypothesis.  We will prove a weaker version of \ref{main1}, from which we get a corollary that is equivalent to 3.2.1 
of \cite{RSW4}.   

\begin{thm}\label{RSW}   Let $\Gamma = \Gamma_x$ be a uniform cyclic ranked poset with $rk_\Gamma(x) = d+1$ and set 
$\Gamma' = \Gamma\setminus\{x\}$.  Let $Z$ be the total space of the order complex $\Delta(\Gamma'\setminus\{*\})$.  Then:

(1)  The cochain epimorphism 
$\Phi_{\Gamma'}: C^n(Z) \to R_{\Gamma'}(n,0)$, as described in the previous section, induces an isomorphism in cohomology in degree $d-1$, that is
$$  H^{d-1}(Z) \cong H_{\Gamma'}(d-1,0).$$
(2)  The map $\tilde H^{d-1}(Z) \to R_\Gamma(d,0)$ given by 
$$[(b_0,\cdots,b_{d-1})] \mapsto r_x\Phi_{\Gamma'}(b_0,\cdots,b_{d-1}) = r_xr_{b_{d-1}}\cdots r_{b_0}$$
is an isomorphism.   
\end{thm}

\begin{proof}
We prove both parts of the theorem by induction on $d$. The case $d=0$ is trivial, as there is nothing to prove.  The case $d=1$ is essentially trivial (and is anyways covered by \ref{magic}). We assume $d\ge 2$.    

Let $\Gamma'(d) = \{y_1,\ldots, y_s\}$ be the elements of rank $d$.  Set $\Omega' = \Gamma' \setminus\{y_1,\ldots, y_s\}$.   
Let $\Omega$ be the poset obtained by adjoining to $\Omega'$ a unique new maximal element $\bar 1$ (so that, in 
particular, $\Omega' = \Omega\setminus\{\bar 1\})$.  In order to apply induction to the pair ($\Omega$,$\Omega'$) we need to 
observe that $\Omega$ is a uniform cyclic ranked poset.  It is ranked because every maximal element of $\Omega'$ has 
rank $d-1$.  It is cyclic by construction.  The fact that $\Omega$ is uniform is Lemma 2.3 of \cite{CPS}.   
Let $W = ||\Delta(\Omega'\setminus\{*\})||$.

For each $y_i$, let $Z_i = ||\Delta((*,y_i])||$.  Then $Z_i \cap W = ||\Delta((*,y_i))||$.  Since $(\Gamma')_{y_i} = [*,y_i]$ is uniform and cyclic, we may also apply the inductive hypothesis to the pairs $(Z_i,Z_i\cap W)$. 

Exactly as in the proof of \ref{main1} we have a degree $+1$ cochain isomorphism $\zeta$ between the augmented cochain complex
$\F^s \to \bigoplus_i C^{n-1}(Z_i\cap W)$ and the cochain complex $C^n(Z,W)$.    Also exactly as in that proof we have a commutative diagram of cochain compexes:
\begin{diagram}
0&  \longrightarrow&C(Z,W) & \longrightarrow&C(Z)  & \longrightarrow&C(W) & \longrightarrow & 0 \\
&& \dTo_{\hat\Phi_{\Gamma'}} &&\dTo_{\Phi_{\Gamma'}} && \dTo_{\Phi_{\Omega'}}\\
0& \longrightarrow&K' & \longrightarrow&R_{\Gamma'}(\cdot,0)  & \longrightarrow&R_{\Omega'}(\cdot,0) & \longrightarrow & 0 \\
\end{diagram}
where $K'$ is the cochain complex $0\to R_{\Gamma'}(d-1,0) \to 0$ with the nonzero term in degree $d-1$.  Consider the final terms of the associated diagram of long exact cohomology sequences:
\begin{diagram}
& \cdots  \longrightarrow&H^{d-2}(W) & \longrightarrow&H^{d-1}(Z,W) & \longrightarrow&H^{d-1}(Z)& \longrightarrow\ \  0 \\
&& \dTo_{\Phi_{\Omega'}} &&\dTo_{\hat\Phi_{\Gamma'}} && \dTo_{\Phi_{\Gamma'}}\\
& \cdots\longrightarrow&\,H_{\Omega'}(d-2,0)\, & \longrightarrow&\, R_{\Gamma'}(d-1,0)\,  & \longrightarrow&\, H_{\Gamma'}(d-1,0)\, & \longrightarrow \ \ 0 \\
\end{diagram}
The first downward arrow in this diagram is an isomorphism by induction.  Using the isomorphism $\zeta$ and induction again we have 
$H^{d-1}(Z,W) = \bigoplus_i H^{d-2}(Z_i\cap W)  = \bigoplus_i R_{\Gamma_{y_i}'}(d-1,0) = R_{\Gamma'}(d-1,0)$.  
Hence the second downward map is also an isomorphism.  Hence the final downward map is also an isomorphism.

This completes the inductive proof of part (1) of the theorem.  Part (2) follows immediately from part (1), exactly as in the proof of 
part (3) of Lemma \ref{magic}.
\end{proof}

\begin{defn}  For each $k\ge 0$ we set $\Gamma^{>k} = \{y\in\Gamma\,|\, rk_\Gamma(y)>k\} \cup \{*\}$.  
\end{defn}

We note that $\Gamma^{>0} = \Gamma$ and that $\Gamma^{>k}$ is uniform if 
$\Gamma$ is uniform. The rank function on 
$\Gamma^{>k}\setminus\{*\}$ is $rk_{\Gamma^{>k}}(a) = rk_\Gamma(a)-k$.   It is also clear that for all $0\le j\le n$, 
$R_{\Gamma^{>k}}(n,j) = R_\Gamma(n+k,j+k)$ and furthermore that  $H_{\Gamma^{>k}}(n,j) = H_\Gamma(n+k,j+k)$.  
Using this notation we get the following corollary to Theorem \ref{RSW}.  This corollary is an exact restatement of Proposition 3.2.1 
from \cite{RSW4} in our notation. (Remark: there is a typographical error in 3.2.1 of \cite{RSW4}, which uses $H^{n-2}$
 instead of $\tilde H^{n-2}$. It is clear that reduced cohomology was intended by the authors.)

\begin{cor}\label{dimension}  Let $\Gamma$ be a finite uniform ranked poset and $*\ne v\in \Gamma$ an element of rank $d+1$.  Then for any $0\le k\le d-1$, 
$$\dim(r_v R_\Gamma(d-1,k)) = \dim(\tilde H^{d-k-1}(\Delta(\Gamma^{>k}_v\setminus\{*,v\})).$$ 
\end{cor}

\begin{proof}
$r_vR_\Gamma(d-1,k) = R_{\Gamma_v}(d,k) = R_{\Gamma^{>k}_v}(d-k,0)$.  Apply (2) of \ref{RSW}.
\end{proof}

As in \cite{RSW4}, this formula easily yields a closed formula for the Hilbert series of $R_\Gamma$. 
We will return to that formula in Section 8.

\section{The Second Main Theorem}

\begin{defn} Let $\Gamma$ be a finite ranked poset.  For any $a<b$ let $X_\Gamma(a,b) = X(a,b)$ be the total space of the order complex $\Delta((a,b))$.
\end{defn}

We note that the dimension of $X(a,b)$ is $rk_\Gamma(b) - rk_\Gamma(a)-2$.  This is consistent with  the definition: $\dim(\Delta(\emptyset)) = -1$.  We also 
take as a (standard) convention $\tilde H^n(\Delta(\emptyset)) = 0$ for $n\ne -1$ and $\tilde H^{-1}(\Delta(\emptyset)) = \F$ 
(cf.~\cite{Woodcock}).  We are now prepared 
to restate and then prove Theorem \ref{intro2}.

\begin{thm}\label{main2}
Let $\Gamma$ be a finite ranked poset.   Then $\Gamma$ is uniform and the algebra $R_\Gamma$ is Koszul if and only if 
$$(*) \qquad \tilde H^n(X(a,b)) = 0 \hbox{ for all } a<b\in \Gamma \hbox{ and all\, } n\ne dim(X(a,b)).$$
\end{thm}
 
\begin{proof}

The proof of \ref{main2} will proceed by induction on the maximal rank of any element of $\Gamma$.  We first prove three technical lemmas. 

\begin{lemma}\label{lemma1}  If the uniform ranked poset $\Gamma$ is Koszul, then the poset $\Gamma^{>k}$ is Koszul for all 
$k\ge 0$. 
\end{lemma}

\begin {proof}
Fix $k\ge 0$ and suppose $\Gamma^{>k}$ is not Koszul.  Then for some $x\in \Gamma^{>k}$,  $(\Gamma^{>k})_x$ is not 
Koszul.  Choose such an $x$ of minimal rank.   

 Note that $(\Gamma^{>k})_x  = (\Gamma_x)^{>k}:= \Gamma_x^{>k}$.   By minimality of $x$, $\Gamma_y^{>k}$ is Koszul
 for every $y<x$ in $\Gamma^{>k}$.  So by \ref{cyclic}, $H_{\Gamma_x^{>k}}(n,j) \ne 0$ for some $0\le j<n $.   But 
  $H_{\Gamma_x^{>k}}(n,j) = H_{\Gamma_x}(n+k,j+k)$.   This contradicts Lemma 2.7, since $\Gamma_x$ is Koszul. 
\end{proof}

\begin{lemma}\label{lemma2} Let $\Gamma=\Gamma_b$ be a cyclic uniform ranked poset.  Set $\Gamma'=\Gamma\setminus\{b\}$, 
$\Omega = \Gamma^{>1}$ and $\Omega' = \Omega\setminus\{b\}$.  Let $Y$ and $Z$ be the total spaces of the order complexes of 
$\Gamma'$ and $\Omega'$ respectively.  Then for all $n>0$,
$$H^n(Y,Z) = \bigoplus_{a\in \Gamma(1)} \tilde H^{n-1}(X(a,b)).$$
\end{lemma}

\begin{proof}
An $n$-cell $\beta = (b_0,\cdots,b_n)$ is a basis element of $C^n(Y,Z)$ if and only if it is not an $n$-cell of $Z$, which happens 
precisely when $b_0$ has rank 1 in $\Gamma$.  For each $a\in \Gamma(1)$ let $C_a^n$ be the $F$-span of those $\beta$ for which
$b_0=a$.  Since each such $a$ is minimal, $d_Y:C_a^n \to C_a^{n+1}$.  Thus we have a cochain complex decomposition: 
$C^*(Y,Z) = \bigoplus_{a\in \Gamma(1)} C_a^*$.  

For $n>0$, let $\zeta: C_a^n\to C^{n-1}(X(a,b))$ be the isomorphism defined by 
$(a,b_1,\cdots,b_n) \mapsto (b_1,\cdots,b_n)$.  Similarly define $\zeta: C_a^0\to \F$ by $(a) \mapsto 1$.  It is clear that we have defined a degree -1 cochain complex isomorphism between $C_a^*$ and the augmented cocomplex $\F \to C^*(X(a,b))$.

The statement of the lemma is now clear. 
\end{proof}

Our last lemma relates the property of being uniform to the property that the topological spaces $X(a,b)$ are connected.

\begin{lemma}\label{lemma3} Let $\Gamma$ be a finite ranked poset which satisfies:
$$(**)\qquad X(a,b) \hbox{ is connected for all } a<b \hbox{ with } rk_\Gamma(b)-rk_\Gamma(a) \ge 3.$$
Then for all $a<b$, the interval $[a,b]$ is uniform (as a ranked poset with unique minimal element $a$).  In particular $\Gamma$ is uniform.

\end{lemma}

\begin{proof}
Choose $a<b$ in $\Gamma$.  We proceed by induction on 
$rk_\Gamma(b)-rk_\Gamma(a)$.  We may also assume 
$rk_\Gamma(b)-rk_\Gamma(a) \ge 3$, since otherwise $[a,b]$ is 
automatically uniform.  

For any $a<c<b$, the interval $[a,c]$ is uniform by induction.  Therefore, it 
remains only to check the definition of uniform against the element $b$ itself.
Returning to the definition of uniform (relative to $[a,b]$) we define 
$S_{[a,b]}(1) = \{ c \,|\, a<c<b \hbox{ and } rk_\Gamma(b)-rk_\Gamma(c)=1\}$.  
The equivalence relation on $S_{[a,b]}(1)$ is defined by transitive extension 
from the definition: $c_1\sim_{[a,b]} c_2$ if there exists $a\le u$ with $u<c_1$, $u<c_2$ and 
$rk_\Gamma(b)-rk_\Gamma(u) = 2$. Let $[c_1], [c_2],\ldots [c_r]$ be the distinct equivalence classes of $\sim_{[a,b]}$. 

It remains only to prove $r=1$, so let us assume $r>1$.  For each $1\le i\le r$, let 
$U_i = (a,[c_i]]:= \{x\in \Gamma\, |\, a<x\le f \hbox{ for some } f\in [c_i]\}$.  Then 
$(a,b) = \cup_i U_i$.  

Since $X(a,b)$ is connected, the poset $(a,b)$ must be connected as a graph.  Since each $U_i$ is a union of maximal intervals in (a,b), the various $U_i$ can not all be disjoint. So we may assume 
$U_1\cap U_2 \not= \emptyset$.  
Choose $y\in U_1\cap U_2$.  Then $y<c_1'$ and $y<c_2'$ for some $c_1'\in [c_1]$ 
and $c_2'\in [c_2]$.  By induction, $[y,b]$ is uniform.  This implies 
$c_1' \sim_{[y,b]} c_2'$, which is clearly a contradiction to $c_1'\not\sim_{[a,b]}c_2'$, since $[y,b] \subset [a,b]$.   Hence $r=1$ and $[a,b]$ is uniform.
\end{proof}

We can now complete the proof of \ref{main2}. 

We first prove the claim that if $\Gamma$ is uniform and Koszul then condition (*) holds.  

Assume $\Gamma$ is uniform and Koszul.  Fix $a<b$ in $\Gamma$. Let $k$ be the rank of 
$a$ and $d+1$ be the rank of $b$.  Then $dim(X(a,b)) =d-k-1$ and so there is nothing to prove unless $d>k+1$.  Since $\Gamma_b$ is uniform and Koszul and $[a,b)$ is contained in $\Gamma_b$, we may assume $\Gamma = \Gamma_b$.  

Let $\Gamma'$, $\Omega$, $\Omega'$, $Y$ and $Z$ be as in Lemma
\ref{lemma2}.  By \ref{lemma1},  $\Omega$ and $\Omega'$ are Koszul.

If $k=0$, that is $a=*$, then $X(a,b) = Y$ and by (3) of \ref{magic}  we have
$\tilde H^n(X(a,b)) = 0$ for all $n<  dim(X(a,b))$.  

Assume $k=1$.  Consider the short exact sequence of cochain complexes
$$0\to C(Y,Z) \to C(Y) \to C(Z) \to 0$$
and associated long exact sequence
$$(***)\qquad \cdots\to \tilde H^{n-1}(Z) \to H^n(Y,Z)\to \tilde H^n(Y) \to \cdots$$
By \ref{magic}, $\tilde H^n(Y) = 0$ for $n<dim(Y) = d-1$ and $\tilde H^{n-1}(Z) = 0$ 
for $n-1<dim(Z)=d-2$.  Hence $H^n(Y,Z)=0$ for $n<d-1$.  By \ref{lemma2}, $\tilde H^{n-1}(X(a,b))$ is a summand of 
$H^n(Y,Z)$ and we obtain $\tilde H^{n-1}(X(a,b))=0$ for $n-1<dim(X(a,b)) = d-2$, as required.

Finally, consider the case $k>1$.  In this case $a\in \Omega$ with 
$rk_\Omega(a)>0$.  By induction on $d$, we immediately get 
$\tilde H^n(X(a,b)) = 0$ for $n<dim(X(a,b))$.  This completes the proof of the first half.

We now turn to the converse. Assume condition (*) holds.  By Lemma \ref{lemma3}, $\Gamma$ is uniform.  We proceed to prove that $\Gamma$ is Koszul, again by induction on 
$d+1$, the maximal rank of any element of $\Gamma$.  If $d=0$ there is nothing to prove.  Assume $d>0$ and choose any element 
$b$ of rank $d+1$.  It suffices to assume $\Gamma=\Gamma_b$.  

Again, let $\Gamma'$, $\Omega$, $\Omega'$ and $Y$ be as 
in Lemma \ref{lemma2}.
By induction, the posets $\Gamma'$, $\Omega$ and $\Omega'$ are all Koszul, since the hypothesis (*) is true for all three and all three only have elements of rank at most $d$. 

For any $*\ne a<b$ in $\Gamma$, $\Gamma_a$ also satisfies (*) and thus, by induction $\Gamma_a$ is Koszul.  Hence by 
\ref{cyclic} it suffices to prove
 $$H_\Gamma(n,k) = 0 \hbox{ for all } 0\le k < n\le d-2.$$

Since $\Omega$ is cyclic, Koszul, and rank $d$, Theorem \ref{cyclic} tells us that for all  $0<k<n\le d-2$:
$$H_\Gamma(n,k) = H_{\Gamma'}(n,k) = H_{\Omega'}(n-1,k-1)=H_\Omega(n-1,k-1)=0.$$  

By Theorem \ref{main1} and (*), for all $0<n\le d-2$, 
$$H_\Gamma(n,0) = H_{\Gamma'}(n,0) = H^n(Y) = 0.$$  

This completes the proof of \ref{main2}.
\end{proof}

\section{A Few Examples}

Given that the Cohen-Macaualy property for posets has been studied extensively for many years, giving specific examples of 
Theorem \ref{intro2} does not seem to be all that productive.  Nevertheless, we discuss three types of examples that were of specific 
interest to us, as they represent different areas where we struggled with the Koszul question in the past. 

\begin{example}   Let $\Gamma$ be the infinite ranked poset of all partitions of non-negative integers.  We identify a partition with 
its Young diagram, and then the order on $\Gamma$ is $\subset$.  The unique minimal element is the partition (0).  For any integer 
$n$ and any partition $\lambda \vdash n$, the finite poset $\Gamma_\lambda$ is a modular lattice, hence Cohen-Macaulay 
(cf.~\cite{BGSsurvey}).  Thus $R_{\Gamma_\lambda}$ is Koszul.  This result was proved earlier by T. Cassidy and Shelton using 
a modification of  the techniques of \cite{CPS}.
\end{example}

\begin{example}  Let ${\mathcal A} = \{H_1,\ldots, H_n\}$ be an arrangement of hyperplanes in an ambient vector space $V$ 
(over an arbitrary field, not necessarily $\F$).   Then the intersection lattice $\Gamma$ associated to ${\mathcal A}$ is semi-modular and therefore Cohen-Macaulay over any field (cf.~\cite{BGSsurvey}).  Hence $R_\Gamma$ is Koszul. This class 
of examples includes, in particular, the lattice of all subspaces of a finite dimensional vector space over a finite field. 
\end{example}

\begin{example}\label{CW}
Let $\Gamma$ be the incidence poset of any finite regular CW complex, let $X$ be the total space of the CW complex and let 
$\bar\Gamma = \Gamma\cup \{\emptyset\}$, where $\emptyset$ is uniquely minimal.   If $\Gamma$ is pure (all maximal cells 
have the same dimension), then we set $\hat \Gamma = \Gamma \cup \{ \emptyset, X\}$ where $X$ is uniquely maximal.  

It is well known that $\bar \Gamma$ is Cohen-Macaulay over any field.  Hence $R_{\bar \Gamma}$ is Koszul over any field.  
This theorem was first proved in both \cite{RSW3} and \cite{CPS}.

The poset $\hat \Gamma$ may well not be Cohen-Macaulay. The second main result of \cite{CPS}, Theorem 5.3, is an exact description, in combinatorial-topological terms, of when the algebras $R_{\hat\Gamma}$ 
are Koszul (these conditions include uniform, expressed as a topological condition).  In the paper \cite{SadSh} it was further shown that the conditions under which $\hat\Gamma$ is Koszul  are topological invariants rather than just combinatorial invariants.  However, the conditions of that theorem can, with some small effort, be translated directly into the statement that $\hat\Gamma$ is Cohen-Macaulay, thereby relating that theorem directly to \ref{main2}. Since Cohen-Macaulay is known to be a topological invariant, Theorem \ref{main2} also recaptures the results of \cite{SadSh}.  

There is an explicit example in \cite{CPS} (Example 5.9 and Theorem 5.10) of a pure 3-dimensional regular CW complex $X$ that 
is contractible, and yet $\hat \Gamma$ is not Koszul.  The proof of this fact was somewhat detailed.  But one can see by inspection, in 
the notation of that example, that the open interval $(C_4, X)$ is not connected as a graph.  That is enough show that $\hat\Gamma$ 
is not Cohen-Macaulay and conclude that the algebra is not Koszul.  
\end{example}

\section{On Numerical Koszulity}

The Hilbert series of an $\N$-graded $\F$-algebra $R = \oplus_iR_i$ is the power series $H(R,t) = \sum_i \dim(R_i)t^i$.  It is well known that if $R$ is a quadratic and Koszul algebra with quadratic dual algebra $R^!$ (cf.~\cite{PP}), then $H(R,-t)H(R^!,t) = 1$.   We say that a quadratic algebra is {\it numerically Koszul} if it satisfies this equation.  

Given a finite uniform ranked poset $\Gamma$, the quadratic dual of the algebra $R_\Gamma$ will be denoted here as 
$A'_\Gamma$.  This algebra was first described in \cite{GRSW} as a deformation of another important quadratic algebra
 $A_\Gamma$, the  {\it splitting algebra} of the poset $\Gamma$.  The algebras $A'_\Gamma$ and $A_\Gamma$ have the 
 same Hilbert series, which was computed in \cite{RSW1} and then recalculated in terms of order complexes in \cite{RSW4}.   We 
 record here Theorem 4.1.1 of \cite{RSW4}, translated into our notation.  (Remark: as with Theorem 3.2.1 of \cite{RSW4}, 
 Theorem 4.1.1 has a typographical error.  The theorem must use reduced cohomology, not cohomology.)

\begin{defn}  The reduced Euler characteristic of a space $X$, relative to $\F$, is 
$$\tilde \chi (X) = \sum_i (-1)^i \dim(\tilde H^i(X))$$
where cohomology is calculated with coefficients in $\F$. 
\end{defn}

Note that $\tilde\chi (\Delta(\emptyset)) = -1$.

Before we state the theorem, we  borrow some nice notation from \cite{RSW4}.

\begin{defn} For any $a\in \Gamma$ and $1\le i\le rk_\Gamma(a)$ we set 
$$\Gamma_{a,i} = \{ w<a \,|\, rk_\Gamma(a)-rk_\Gamma(w) \le i-1\} = \Gamma^{>rk_\Gamma(a)-i} \cap (*,a)$$
\end{defn}

Note that $\Gamma_{a,i}$ is a subposet of $[*,a)$ and the dimension of $\Delta(\Gamma_{a,i})$ is 
$i-2$.  It is helfpul to note that $\Gamma_{a,1} = \emptyset$, $\Gamma_{a,2} = S_1(a)$ and $\Gamma_{a,rk_\Gamma(a)} = (*,a)$.   

\begin{thm}\label{HS1} {\rm (\cite{RSW4}, 4.1.1)} Let $\Gamma$ be a uniform finite ranked poset.  Then:
$$  H(A_\Gamma, t)^{-1} = 1 + \sum_{i\ge 1}\, \sum_{a\in \Gamma \atop rk_\Gamma(a) \ge i} \tilde\chi(\Delta(\Gamma_{a,i})) t^i $$
\end{thm}

For further use we also record the Hilbert series of $R_\Gamma$ from 3.1.1 of \cite{RSW4} (again, correcting for reduced cohomology).

\begin{thm}\label{HS2} {\rm (\cite{RSW4}, 3.1.1)} Let $\Gamma$ be a uniform finite ranked poset.  Then:
$$  H(R_\Gamma, -t) = 1 + \sum_{i\ge 1}\, \sum_{a\in \Gamma \atop rk_\Gamma(a) \ge i} (-1)^{i-2}\dim\tilde H^{i-2}(\Delta(\Gamma_{a,i})) t^i $$
\end{thm}

\begin{defn} Let $\Gamma$ be a finite uniform ranked poset.

(1) Let $v\in \Gamma$ and $i\le rk_\Gamma(v)$.  We say that the pair $(v,i)$ is {\it good} if 
$$ \tilde\chi(\Delta(\Gamma_{v,i})) = (-1)^{i-2}\dim\tilde H^{i-2}(\Delta(\Gamma_{v,i}))$$
and {\it bad} if the equality does not hold. 

(2) We define the {\it numerical Koszul defect} of $\Gamma$ to be 
$$\begin{array}{ll}
 \NKD(\Gamma) &= H(R_\Gamma, -t) - H(A_\Gamma,t)^{-1}\\[8pt]
 & = \sum\limits_{(v,i)\  bad} \left[ (-1)^{i-2}\dim\tilde H^{i-2}(\Delta(\Gamma_{v,i}))
- \tilde\chi(\Delta(\Gamma_{v,i}))\right] t^i\\
\end{array} $$
\end{defn}

Since $A_\Gamma$ and $A_\Gamma' = R_\Gamma^!$ have the same Hilbert series for any uniform $\Gamma$, we see that the ring 
$R_\Gamma$ is numerically Koszul if and only if $NKD(\Gamma) = 0$. 

We record the following trivial observation which we will use frequently:  if $\Gamma$ is uniform, then for  any $v\in \Gamma$, $(v,1)$,
$(v,2)$, $(v,3)$ are always good ($(v,3)$ is good because uniformity implies the graph $\Gamma_{v,3}$ is connected).  And $(v,4)$ is good if and only if $H^1(\Delta(\Gamma_{(v,4})) = 0$.  This immediately gives the following simple theorem.

\begin{thm}\label{rank4}  If $\Gamma$ is a finite uniform ranked poset and no element of $\Gamma$ has rank bigger than 4, then $R_\Gamma$ is Koszul if and only if $R_\Gamma$ is numerically Koszul.
\end{thm}

\begin{proof}
Only one direction needs to be proved. Assume $R_\Gamma$ is numerically Koszul.  Fix any $a<b$ in $\Gamma$.  If $dim(X(a,b))\le 1$ 
then it is clear from uniform that $\tilde H^n(X(a,b))$ is non-zero only for $n=dim(X(a,b))$.  So assume $dim(X(a,b))=2$ 
(the maximum possible such dimension).  This can only happen if $rk_\Gamma(b) = 4$ and $a = *$.  But by hypothesis, $(b,4)$ is 
good. Since $\Gamma_{b,4} = (*,b)$, we have $\tilde H^n(X(*,b)) =0$ for $n\ne 2$.  We have shown $\Gamma$ is Cohen-Macaualay 
and thus $\Gamma$ is Koszul.
\end{proof}

We now describe a construction for combining two uniform finite ranked posets, over which the $\NKD$ will be additive.  
This will allow us to construct examples that are numerically Koszul  but not Koszul.   

For the time being, let $\Gamma$ and $\Omega$ be two uniform finite ranked posets with minimal elements 
$*_\Gamma$ and $*_\Omega$ respectively.   Let $X_\Gamma$ and $X_\Omega$ be the total spaces of the respective order 
complexes $\Delta(\Gamma)$ and $\Delta(\Omega)$.   Fix elements $v\in \Gamma(1)$ and $v'\in \Omega(1)$.

\begin{defn}  With notation as above we set
$$ \Gamma \vee_{(v,v')}\Omega = \Gamma \cup \Omega / ( *_\Gamma \sim *_\Omega, v\sim v')$$
We identify $\Gamma$ and $\Omega$ as subsets of $ \Gamma \vee_{(v,v')}\Omega$.  The set inherits an order from $\Gamma$ and $\Omega$ wherein two elements are related if and only if they are related in either $\Gamma$ or in $\Omega$. We denote the unique minimal element by  $*$ and the common image of $v$ and $v'$ by $\bar v$.  
\end{defn}

It is clear that the poset $ \Gamma \vee_{(v,v')}\Omega$ inherits the uniform property from 
$\Gamma$ and $\Omega$.  The following lemma is obvious. 

\begin{lemma}\label{top}  Let notation be as above.  Then the total space of the order complex of the poset 
$\Gamma\vee_{(v,v')}\Omega$ is $X_\Gamma \vee X_\Omega$.  
\end{lemma}

\begin{lemma}\label{NKD}
Let notation be as above.  Then
$$ \NKD(\Gamma \vee_{(v,v')}\Omega) = \NKD(\Gamma) + \NKD(\Omega).$$
\end{lemma}

\begin{proof}
To ease the notation, let  $\Theta =  \Gamma \vee_{(v,v')}\Omega$.  For any $a\in \Theta$ we see that $\Theta_{a,i} = \Gamma_{a,i}$ if $a\in \Gamma$ and $\Theta_{a,i} = \Omega_{a,i}$ if $a\in \Omega$.  In particular, $\Theta_{\bar v,1} = \emptyset$.  
Thus, by \ref{HS1}  we have the following decomposition:
$$
\begin{array}{lll} 
H(A_\Theta,t)^{-1} &= 1 + \sum\limits_{a\in \Gamma \atop rk_\Gamma(a) \ge i\ge 1} \tilde\chi(\Delta(\Gamma_{a,i})) t^i\\[22pt]
&\qquad\quad + \sum\limits_{a\in \Omega \atop rk_\Gamma(a) \ge i\ge 1} \tilde\chi(\Delta(\Omega_{a,i})) t^i -  \tilde\chi(\Theta_{\bar v,1})t\\[22pt]
&= H(A_\Gamma,t)^{-1} + H(A_\Omega,t)^{-1} + t - 1\\
\end{array}
$$
Exactly the same calculation for $H(R_\Theta,-t)$ yields
$$ H(R_\Theta,-t) = H(R_\Gamma,-t) + H(R_\Omega, -t) + t - 1$$
Subtracting gives the required equation.
\end{proof}

\begin{lemma}\label{S2XI}  
Let $\Gamma$ be the uniform ranked poset shown to the right below.  Then  $\NKD(\Gamma) = -t^5$.  In particular $\Gamma$ is not Koszul.

\begin{center}
\includegraphics[width=2.3in]{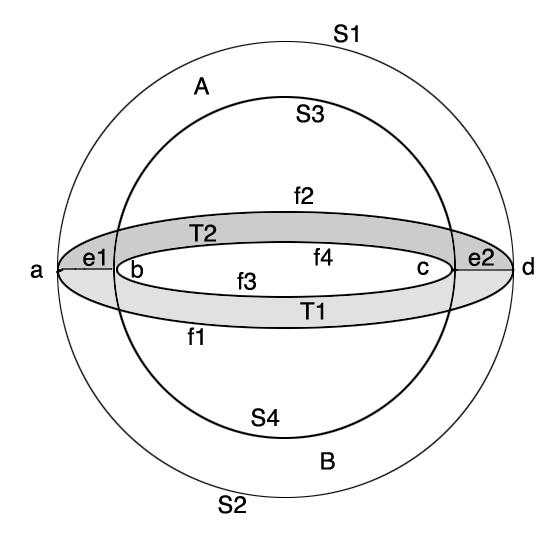}\hskip .2in
\includegraphics[width=2.3in]{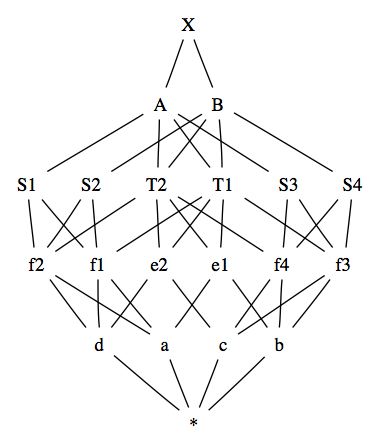}
\end{center}
\end{lemma}

\begin{proof}
The poset $\Gamma$ is of the form $P \cup \{*,X\}$, where $P$ is the incidence poset of a regular CW complex realization of 
$S^2 \times I$. The picture to the left labels the CW complex $P$ (to the best of our abilities).   Because $P$ is a CW complex, 
the only pairs $(v,i)$ that might be bad are the pairs $(X,5)$ and $(X,4)$.  

The pair $(X,5)$ is bad. To see this note $\Gamma_{X,5} = (*,X) = P$ and hence $||\Delta(\Gamma_{X,5})|| = S^2\times I$.   This 3-dimensional space is homotopic to $S^2$ and has
non-zero reduced cohomology only in degree 2. Hence the pair $(X,5)$ is bad and contributes $-t^5$ to $\NKD(\Gamma)$. 

The pair $(X,4)$ is good. To see this, first apply \ref{lemma2} to get 
$$H^n(\Delta(\Gamma_{X,5}), \Delta(\Gamma_{X,4}))= \bigoplus_{v\in \Gamma(1)}\tilde H^{n-1}(\Delta(\,(v,X)\,)).$$  
Since $S^2\times I$ is a manifold with boundary 
and each $0$-cell in $P$ is on the boundary, the spaces $\Delta(\,(v,X)\,)$ are all homeomorphic to $2$-discs and thus have 
no reduced cohomology. I.e. $H^n(\Delta(\Gamma_{X,5}), \Delta(\Gamma_{X,4}))=0$ for all $n$.  Since $\tilde H^n(\Delta(\Gamma_{X,5})) = 0$ for all $n\ne 2$, the usual long exact sequence for relative cohomology tells us $\tilde H^n(\Delta(\Gamma_{X,4})) = 0$ for all $n\ne 2$.  Thus $(X,4)$ is good.

This proves the lemma.
\end{proof}  

\begin{lemma}\label{S1}
Let $\Omega$ be the the uniform ranked poset shown below.  Then $\NKD(\Omega) = t^5$.  In particular $\Omega$ is not Koszul. 
\begin{center}
\includegraphics[width=1in]{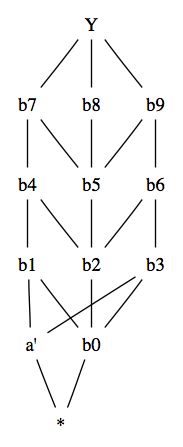}
\end{center}

\end{lemma}

\begin{proof}  We see at once that $\Omega$ is not Koszul because the interval $(a',Y)$ is not connected as a graph.  By 
direct inspection, every pair $(v,4)$ is good, because every pair $(v,4)$ corresponds to a contractible space with no 
reduced cohomology. This leaves $(Y,5)$ as the only pair that can be bad.   Since $\Delta(\Omega_{Y,5}) = \Delta((*,Y))$ is 3-dimensional but homotopic to $S^1$ (by inspection), the pair $(Y,5)$ is bad and contributes exactly $t^5$ to $\NKD(\Omega)$, as claimed. 
 \end{proof} 
 
We now see easily that a numerically Koszul algebra in our class need not be Koszul, as promised in Theorem \ref{intro3}. 

\begin{thm}\label{main3}  Let $\Gamma$ be as in \ref{S2XI} and $\Omega$ as in \ref{S1}.  Set $\Theta = \Gamma \vee_{(a,a')}\Omega$.  Then the algebra $R_\Theta$ is not Koszul, but it is numerically Koszul.
\end{thm}

\begin{proof}
The intervals $(*,Y)$ and $(*,X)$ both show that the poset $\Theta$ is not Cohen-Macauley.  Hence by \ref{main2}, $R_\Theta$ is not Koszul.

By \ref{NKD}, \ref{S2XI} and \ref{S1}, $\NKD(\Theta) = \NKD(\Gamma) + \NKD(\Omega) = -t^5 + t^5 = 0$.  Thus $\Theta$ is numerically Koszul. 
\end{proof}

It is clear that any number of examples could be constructed in this fashion, but the resulting examples are not very satisfying, as they are far from being cyclic.  Fortunately, through much more ad-hoc methods we were able to obtain the following cyclic example.

\begin{thm}\label{wild}
Let $\Gamma$ be the uniform ranked poset shown below.  Then $\Gamma$ is numerically Koszul, but not Koszul.
\begin{center}
\includegraphics[width=4in]{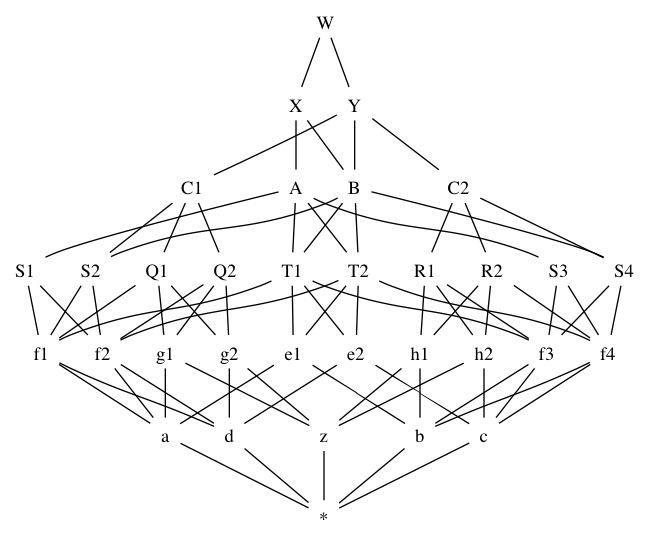}
\end{center}
\end{thm}

\begin{proof}

We first list several useful observations about $\Gamma$. $[*, X]$ is the poset from \ref{S2XI}. In particular, $(*, X)$ is the incidence poset of a regular CW complex realization of $S^2 \times I$; we will refer to this CW complex as $\| X \|$. Let $\| Y \|$ be the three-dimensional CW complex drawn (to the best of our abilities) below. $\| Y \|$ has three $3$-cells $B$, $C1$, and $C2$ and we note that $B$ is also a $3$-cell of $\| X \|$. $\| Y \|$ has a singular point at $z$ and is homotopic to $S^1$. $(*, Y)$ in $\Gamma$ is the incidence poset of $\| Y \|$. $(*, W)$ in $\Gamma$ is not the incidence poset of a CW complex. 

\begin{center}
\includegraphics[width=4in]{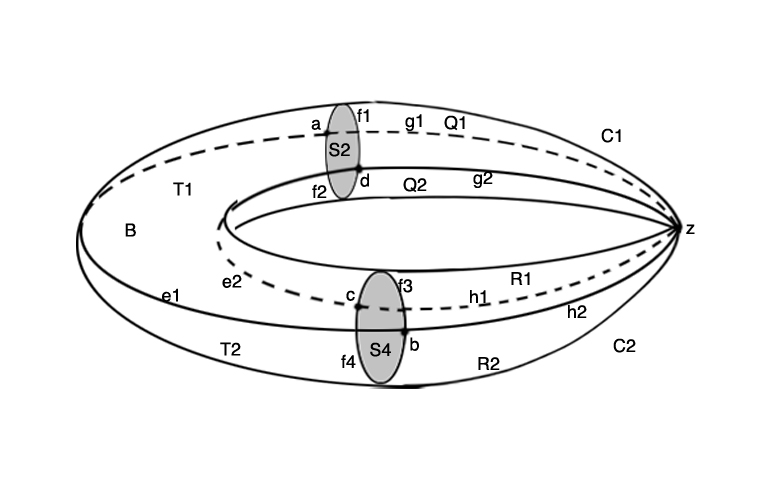}
\end{center}

By construction, $\Gamma$ is uniform. $\Gamma$ is not Cohen-Macaulay because the open interval $(z, Y)$ is not connected as a graph, By \ref{main2}, $R_\Gamma$ is not Koszul.

To see that $\Gamma$ is numerically Koszul, we need to examine the pairs $(A,4)$, $(B,4)$, $(C1,4)$, $(C2, 4)$, $(X,4)$, $(X,5)$, $(Y,4)$, $(Y, 5)$, $(W, 4)$, $(W, 5)$, and $(W, 6)$. 

Due to the regular CW structures, the realizations of the order complexes of $\Gamma_{A, 4}$, $\Gamma_{B, 4}$, $\Gamma_{C1, 4}$, and $\Gamma_{C2, 4}$ are each homeomorphic to $2$-spheres. We conclude $(A,4)$, $(B,4)$, $(C1,4)$, and $(C2, 4)$ are good. \ref{S2XI} tells us $(X, 4)$ is good and $(X, 5)$ is bad. $(Y, 5)$ is bad because $\| \Delta(\Gamma_{Y,5})\|$ is homotopic to $S^1$. 

Before calculating the remaining pairs, we will state a useful observation. Let $V$ be a topological space and let $U_1, U_2$ be closed subsets of $V$ such that $U_1 \cup U_2 = V$. If $U_1 \cap U_2$ is contractible, then $\text{cone}(U_1) \cup \text{cone}(U_2)$ is contractible. 

Claim: $(Y, 4)$ is good. We apply \ref{lemma2} to get 
$$(**)  \qquad H^n(\Delta(\Gamma_{Y,5}), \Delta(\Gamma_{Y, 4}))= \bigoplus\limits_{t\in \Gamma_Y(1)} \tilde H^{n-1}(\Delta((t,Y)))$$
for all $n \geq 0$. The intervals $(a, Y), (d, Y), (b, Y)$ and $(c, Y)$ are isomorphic as posets and we can use the above useful observation to see that the order complex of each interval is contractible. 
The realization of the order complex of $(z, Y)$ is homotopic to $S^0$. From (**), we conclude that $H^n(\Delta(\Gamma_{Y, 5}), \Delta(\Gamma_{Y, 4}))$ is one-dimensional for $n = 1$ and is zero otherwise. 

We now apply the standard long exact cohomology sequence related to relative cohomology for $\Delta(\Gamma_{Y, 5})$ and $\Delta(\Gamma_{Y, 4})$. Recalling that $H^n(\Delta(\Gamma_{Y, 5}))$ is one-dimensional for $n = 1$ and is zero otherwise, we see that $H^1(\Delta(\Gamma_{Y, 4})) = 0$ as required. 

$(W, n)$ is good for all $4 \leq n \leq 6$ because $\| \Delta(\Gamma_{W, n}) \|$ is contractible by the above useful observation.

Finally, we compute
$$\begin{array}{ll}
 \NKD(\Gamma) & = \sum\limits_{(V,i)\  bad} \left[ (-1)^{i-2}\dim\tilde H^{i-2} (\Delta(\Gamma_{V,i})) - \tilde\chi(\Delta(\Gamma_{V,i}))\right] t^i \\[15pt]
 & = [(-1)^{3}\dim\tilde H^{3}(\Delta(\Gamma_{X,5}))
- \tilde\chi(\Delta(\Gamma_{X,5})) ] t^5 \\[10pt]
 & + [(-1)^{3}\dim\tilde H^{3}(\Delta(\Gamma_{Y,5}))
- \tilde\chi(\Delta(\Gamma_{Y,5}))] t^5 \\[10pt]
& = [- \tilde\chi(\Delta(\Gamma_{X,5})) - \tilde\chi(\Delta(\Gamma_{Y,5}))]t^5 = [(-1) - (-1)] t^5 = 0.
\end{array} $$

Thus $R_{\Gamma}$ is numerically Koszul and this completes our proof. 
\end{proof}

\bibliographystyle{amsplain}
\bibliography{bibliog}

\providecommand{\bysame}{\leavevmode\hbox to3em{\hrulefill}\thinspace}
\providecommand{\MR}{\relax\ifhmode\unskip\space\fi MR }
\providecommand{\MRhref}[2]{%
  \href{http://www.ams.org/mathscinet-getitem?mr=#1}{#2}
}
\providecommand{\href}[2]{#2}
\begin{thebibliography}{10}

\bibitem{BGSsurvey}
A.~Bj{\"o}rner, A.~M. Garsia, and R.~P. Stanley, \emph{An introduction to
  {C}ohen-{M}acaulay partially ordered sets}, Ordered sets ({B}anff, {A}lta.,
  1981), NATO Adv. Study Inst. Ser. C: Math. Phys. Sci., vol.~83, Reidel,
  Dordrecht, 1982, pp.~583--615. \MR{661307 (83i:06001)}

\bibitem{CPS}
Thomas Cassidy, Christopher Phan, and Brad Shelton, \emph{Noncommutative
  {K}oszul algebras from combinatorial topology}, J. Reine Angew. Math.
  \textbf{646} (2010), 45--63. \MR{2719555}

\bibitem{GGRSW}
Israel Gelfand, Sergei Gelfand, Vladimir Retakh, Shirlei Serconek, and
  Robert~Lee Wilson, \emph{Hilbert series of quadratic algebras associated with
  pseudo-roots of noncommutative polynomials}, J. Algebra \textbf{254} (2002),
  no.~2, 279--299. \MR{MR1933871 (2003k:16039)}

\bibitem{GRSW}
Israel Gelfand, Vladimir Retakh, Shirlei Serconek, and Robert~Lee Wilson,
  \emph{On a class of algebras associated to directed graphs}, Selecta Math.
  (N.S.) \textbf{11} (2005), no.~2, 281--295. \MR{MR2183849 (2006i:05071)}

\bibitem{PP}
Alexander Polishchuk and Leonid Positselski, \emph{Quadratic algebras},
  University Lecture Series, vol.~37, American Mathematical Society,
  Providence, RI, 2005. \MR{MR2177131}

\bibitem{Polo}
Patrick Polo, \emph{On {C}ohen-{M}acaulay posets, {K}oszul algebras and certain
  modules associated to {S}chubert varieties}, Bull. London Math. Soc.
  \textbf{27} (1995), no.~5, 425--434. \MR{1338684 (96m:20068)}

\bibitem{Reisner}
Gerald~Allen Reisner, \emph{Cohen-{M}acaulay quotients of polynomial rings},
  Advances in Math. \textbf{21} (1976), no.~1, 30--49. \MR{0407036 (53
  \#10819)}

\bibitem{RSW4}
Vladimir Retakh, Shirlei Serconek, and Robert Wilson, \emph{Hilbert series of
  algebras associated to directed graphs and order homology}, Journal of Pure
  and Applied Algebra \textbf{216} (2012), no.~6, 1397 -- 1409.

\bibitem{RSW1}
Vladimir Retakh, Shirlei Serconek, and Robert~Lee Wilson, \emph{Hilbert series
  of algebras associated to directed graphs}, J. Algebra \textbf{312} (2007),
  no.~1, 142--151. \MR{MR2320451 (2008f:16088)}

\bibitem{RSW3}
\bysame, \emph{Koszulity of splitting algebras associated with cell complexes},
  J. Algebra \textbf{323} (2010), no.~4, 983--999. \MR{2578588 (2011e:16055)}

\bibitem{SadSh}
Hal Sadofsky and Brad Shelton, \emph{The {K}oszul property as a topological
  invariant and measure of singularities}, Pacific J. Math. \textbf{252}
  (2011), no.~2, 473--486. \MR{2860435}

\bibitem{Stanley75}
Richard~P. Stanley, \emph{Cohen-{M}acaulay rings and constructible polytopes},
  Bull. Amer. Math. Soc. \textbf{81} (1975), 133--135. \MR{0364231 (51 \#486)}

\bibitem{Woodcock}
D.~Woodcock, \emph{Cohen-{M}acaulay complexes and {K}oszul rings}, J. London
  Math. Soc. (2) \textbf{57} (1998), no.~2, 398--410. \MR{1644229 (99g:13025)}

\end{thebibliography}
\end{document}